\RequirePackage{etex}
\documentclass[a4paper]{article}
\usepackage{amsmath,amssymb,amsthm,enumerate}
\usepackage{microtype}
\usepackage[top=3cm,bottom=3cm,left=2.5cm,right=2.5cm]{geometry}
\usepackage{tikz}
\usetikzlibrary{arrows,cd}
\usepackage[hypertexnames=false]{hyperref}
\hypersetup{
  colorlinks=true,
	citecolor=red,
	linkcolor=blue,
	urlcolor=orange
	}

\numberwithin{equation}{section}

\newtheorem{theorem}{Theorem}[section]

\theoremstyle{definition}
\newtheorem{remark}[theorem]{Remark}

\newtheorem{example}[theorem]{Example}

\newtheorem{definition-theorem}[theorem]{Definition-Theorem}
\newtheorem{question}[theorem]{Question}

\newcommand{\rad}{\operatorname{\mathrm{rad}}}

\newcommand{\Ext}{\operatorname{\mathrm{Ext}}}

\newcommand{\add}{\operatorname{\mathrm{add}}}
\newcommand{\Ind}{\operatorname{\mathrm{Ind}}}
\newcommand{\Res}{\operatorname{\mathrm{Res}}}

\newcommand{\setmid}{\; \middle|\;}
\newcommand{\inertiagp}{I}
\newcommand{\sttilt}{\operatorname{\mathrm{s\tau-tilt}}}
\newcommand{\stautilt}{\operatorname{\mathrm{s\tau-tilt}}}
\newcommand{\taurigid}{\operatorname{\mathrm{\tau-rigid}}}
\newcommand{\lmod}{\operatorname{\mathrm{\hspace{-2pt}-mod}}}

\newcommand{\Addresses}{{
  \bigskip
  \footnotesize

  Ryotaro~KOSHIO\,
  \par\nopagebreak
  \textsc{Department of Mathematics, Tokyo University of Science}
  \par\nopagebreak
	1-3, Kagurazaka, Shinjuku-ku, Tokyo, 162-8601, Japan
  \par\nopagebreak
  E-mail: \href{mailto:1120702@ed.tus.ac.jp}{1120702@ed.tus.ac.jp}

  \medskip

  Yuta~KOZAKAI\,
  \par\nopagebreak
  \textsc{Department of Mathematics, Tokyo University of Science}
	\par\nopagebreak
	1-3, Kagurazaka, Shinjuku-ku, Tokyo, 162-8601, Japan
	\par\nopagebreak
  E-mail: \href{mailto:kozakai@rs.tus.ac.jp}{kozakai@rs.tus.ac.jp}
}}

\begin{document}

\title{
  A characterization for induced modules to be support $\tau$-tilting modules
  \footnote{\emph{Mathematics Subject Classification} (2020). 20C20, 16G10.}
  \footnote{\emph{Keywords.} Support \(\tau\)-tilting modules, blocks of finite groups, induction functors}
}
\author{Ryotaro~KOSHIO \and Yuta~KOZAKAI}

\maketitle
\begin{abstract}
  Let $\tilde{G}$ be a finite group and $G$ a normal subgroup of $\tilde{G}$. In this paper, we give a necessary and sufficient condition for $\Ind_G^{\tilde{G}}M$ to be a support \(\tau\)-tilting $k\tilde{G}$-module for a $kG$-module $M$. Moreover, we give the block version of the result.
\end{abstract}

\section{Introduction}\label{introduction}
Support \(\tau\)-tilting modules introduced in \cite{MR3187626} play important roles in the representation theory of finite-dimensional algebras.
In fact, they correspond to various representation theoretical objects such as two-term silting complexes, functorially finite torsion classes, left finite semibricks, two-term simple-minded collections and more (see \cite{MR3187626, MR4139031, MR3220536, MR3178243}).

On the other hand, the classifications and characterizations of support \(\tau\)-tilting modules over group algebras are useful tools for solving Brou\'{e}'s Abelian Defect Group Conjecture.
In fact, as mentioned above, support \(\tau\)-tilting modules are in bijection with two-term silting complexes and simple-minded collections.
Moreover, two-term silting complexes over group algebras (or blocks of finite groups) are consistent with two-term tilting complexes, and simple-minded collections are important for Okuyama-methods, which is very important to consider the conjecture (see \cite{okuyama1997some,MR1947972}).

Thus, the authors have given methods to get support \(\tau\)-tilting modules over group algebras or blocks of finite groups \cite{10.55937/sut/1670501315,KK2,MR4243358}.
The methods taken in these contexts are obtaining support \(\tau\)-tilting \(k\tilde{G}\)-modules from support \(\tau\)-tilting \(kG\)-modules using the induction functor \(\Ind_G^{\tilde{G}}:=k\tilde{G}\otimes_{kG}\bullet\), where \(G\) is a normal subgroup of a finite group \(\tilde{G}\).
Naturally, the following question arises:
\begin{question}
  What are classes of the \(kG\)-modules which are
  sent to the support \(\tau\)-tilting \(k\tilde{G}\)-modules by the induction functor \(\Ind_G^{\tilde{G}}\)?
\end{question}

In other words, we are interested in
clarifying the set
\[
  (\Ind_G^{\tilde{G}})^{-1}(\sttilt k\tilde{G}):=\left\{ M\in kG\lmod \setmid \Ind_G^{\tilde{G}}M\in \sttilt k\tilde{G} \right\}/=_{\add},
\]
here, \(M=_{\add}M'\) means \(\add M=\add M'\) for \(M, M'\in kG\lmod\), which gives an equivalent relation on \(kG\lmod\).
In \cite{10.55937/sut/1670501315,KK2,MR4243358}, some subsets in the set \((\Ind_G^{\tilde{G}})^{-1}(\sttilt k\tilde{G})\) are given under some assumptions, but the set is never clarified completely.
In this paper, we give the complete answer of the question for any finite group \(\tilde{G}\) and its normal subgroup \(G\).
\begin{theorem}\label{main-theorem-1}
  Let \(G\) be a normal subgroup of a finite group \(\tilde{G}\).
  For a \(kG\)-module \(M\), the induced module \(\Ind_G^{\tilde{G}}M\) is a support \(\tau\)-tilting \(k\tilde{G}\)-module if and only if \(M\) is a \(\tau\)-rigid \(kG\)-module such that \(\bigoplus_{\tilde{g}\in[\tilde{G}/G]}\tilde{g}M\) is a support \(\tau\)-tilting \(k\tilde{G}\)-modules, that is, it holds that
  \[
    (\Ind_G^{\tilde{G}})^{-1}(\sttilt k\tilde{G})
    =\left\{ M\in \taurigid kG \setmid \bigoplus_{\tilde{g}\in [\tilde{G}/G]}\tilde{g}M \in \sttilt kG \right\}.
  \]
\end{theorem}

As a natural question, we wonder if we get the block version of the above theorem for group algebras.
That is, we are interested in clarifying the set
\[
  (\tilde{B}\Ind_G^{\tilde{G}})^{-1}(\sttilt \tilde{B}):=\left\{ M\in B\lmod \setmid \tilde{B}\Ind_G^{\tilde{G}}M\in \sttilt \tilde{B} \right\}/=_{\add}.
\]
As a result, we get the block version of the theorem.
\begin{theorem}\label{main-theorem-2}
  Let \(G\) be a normal subgroup of a finite group \(\tilde{G}\), \(B\) a block of \(kG\) and \(\tilde{B}\) a block of \(k\tilde{G}\) covering \(B\).
  For a \(B\)-module \(M\), the module \(\tilde{B}\Ind_G^{\tilde{G}}M\) is a support \(\tau\)-tilting \(\tilde{B}\)-module if and only if \(M\) is a \(\tau\)-rigid \(B\)-module such that \(\bigoplus_{\tilde{g}\in[I_{\tilde{G}}(B)/G]} \tilde{g}M\) is a support \(\tau\)-tilting \(B\)-modules, that is, it holds that
  \[(\tilde{B}\Ind_G^{\tilde{G}})^{-1}(\sttilt \tilde{B})
    =\left\{ M\in \taurigid B \setmid \bigoplus_{\tilde{g}\in [I_{\tilde{G}}(B)/G]}\tilde{g}M \in \sttilt B \right\},\]
  where \(I_{\tilde{G}}(B)\) is the inertial group of \(B\) in \(\tilde{G}\).
\end{theorem}

Throughout this paper, we use the following notation and terminologies.
Modules mean left modules.
The symbol \(k\) means an algebraically closed field of characteristic \(p>0\).
Let \(\Lambda\) be a finite-dimensional algebra.
For a \(\Lambda\)-module \(M\), we denote by \(\tau M\) the Auslander-Reiten translate of \(M\) and by \(\add M\) the full subcategory of the \(\Lambda\)-module category \(\Lambda\lmod\) whose objects are finite direct sums of direct summands of \(M\).
For \(\Lambda\)-modules \(M\) and \(M'\), we write \(M=_{\add} M'\) if \(\add M =\add M'\).
This relation is an equivalence relation, and we call this relation additive equivalence relation.
We denote by \(\stautilt \Lambda\) the set of additive equivalence classes of support \(\tau\)-tilting \(\Lambda\)-modules and by \(\taurigid \Lambda\) the one of \(\tau\)-rigid \(\Lambda\)-modules.
Let \(\tilde{G}\) be a finite group and \(G\) a normal subgroup of \(\tilde{G}\).
We denote \([\tilde{G}/G]\) a set of representatives of \((\tilde{G}/G)\).
For a \(kG\)-module \(M\) and \(\tilde{g}\in \tilde{G}\), we define a \(kG\)-module \(\tilde{g}M\) consisting of symbols \(\tilde{g}m\) as a set, where \(m\in M\),
and its \(kG\)-module structure is given by \(\tilde{g}m+\tilde{g}m':=\tilde{g}(m+m')\), \(\lambda(\tilde{g}m)=\tilde{g}(\lambda m)\) and \(g(\tilde{g}m):=\tilde{g}(\tilde{g}^{-1}g\tilde{g}m)\) for \(m, m'\in M, g\in G\) and \(\lambda\in k\).
We say a \(kG\)-module \(M\) is \(\tilde{G}\)-invariant if \(M\) is isomorphic to \(\tilde{g}M\) for any \(\tilde{g}\in \tilde{G}\).
For a block \(B\) of \(kG\), we denote by \(\inertiagp_{\tilde{G}}(B):=\left\{ \tilde{g}\in \tilde{G} \setmid \tilde{g}B\tilde{g}^{-1}=B \right\}\) the inertial group of \(B\) in \(\tilde{G}\).

\section{Proofs of the theorems}\label{proof}
We now give a proof of Theorem \ref{main-theorem-1}.
\begin{proof}[Proof of Theorem \ref{main-theorem-1}]
  Let \(M\) be a \(kG\)-module satisfying that \(\Ind_G^{\tilde{G}}M\) is a support \(\tau\)-tilting \(k\tilde{G}\)-module.
  Then it holds that \(\Res_G^{\tilde{G}}\Ind_G^{\tilde{G}}M\cong\bigoplus_{\tilde{g}\in[\tilde{G}/G]}\tilde{g}M\) by Mackey's decomposition formula (see \cite[Lemma 8.7.]{MR860771}).
  Moreover, \(\Ind_G^{\tilde{G}}M\) is relatively \(G\)-projective and satisfies that
  \[
    \Ind_G^{\tilde{G}}\Res_G^{\tilde{G}}\Ind_G^{\tilde{G}}M
    \cong\Ind_G^{\tilde{G}}(\bigoplus_{\tilde{g}\in[\tilde{G}/G]}\tilde{g}M)
    =_{\add}\Ind_G^{\tilde{G}}M.
  \]
  Hence, the restricted module \(\Res_G^{\tilde{G}}\Ind_G^{\tilde{G}}M\) is a support \(\tau\)-tilting \(kG\)-module by \cite[Theorem 1.2.]{https://doi.org/10.48550/arxiv.2301.04963}.
  In particular, the module \(M\) is a direct summand of the support \(\tau\)-tilting \(kG\)-module \(\Res_G^{\tilde{G}}\Ind_G^{\tilde{G}}M \cong \bigoplus_{\tilde{g}\in[\tilde{G}/G]}\tilde{g}M\).
  Therefore, the \(kG\)-module \(M\) is a \(\tau\)-rigid \(kG\)-module such that \(\bigoplus_{\tilde{g}\in[\tilde{G}/G]}\tilde{g}M\) is a support \(\tau\)-tilting \(kG\)-module.

  On the other hand, let \(M\) be a \(\tau\)-rigid \(kG\)-module such that \(\bigoplus_{\tilde{g}\in[\tilde{G}/G]}\tilde{g}M\) is a support \(\tau\)-tilting \(kG\)-module.
  By Mackey's decomposition formula again, \(\Res_G^{\tilde{G}}\Ind_G^{\tilde{G}}M \cong \bigoplus_{\tilde{g}\in[\tilde{G}/G]}\tilde{g}M\) is a support \(\tau\)-tilting \(kG\)-module,
  and it is a \(\tilde{G}\)-invariant \(kG\)-module clearly.
  Hence, the induced module \(\Ind_G^{\tilde{G}}\Res_G^{\tilde{G}}\Ind_G^{\tilde{G}}M\) is a support \(\tau\)-tilting \(k\tilde{G}\)-module by \cite[Theorem 3.2.]{10.55937/sut/1670501315}.
  Moreover, it holds that
  \[
    \Ind_G^{\tilde{G}}M =_{\add}
    \Ind_G^{\tilde{G}}(\bigoplus_{\tilde{g}\in[\tilde{G}/G]}\tilde{g}M) =_{\add}
    \Ind_G^{\tilde{G}}\Res_G^{\tilde{G}}\Ind_G^{\tilde{G}}M.
  \]
  Therefore, \(\Ind_G^{\tilde{G}}M\) is a support \(\tau\)-tilting \(k\tilde{G}\)-module.
\end{proof}
Let \(\Lambda\) be a finite-dimensional algebra.
For \(\Lambda\)-modules \(M\) and \(M'\), we write \(M \leq_{\add} M'\) if \(\add M\subseteq \add M'\).
This relation is clearly reflexive and transitive.
Moreover, if \(M \leq_{\add} M'\) and \(M' \leq_{\add} M\) then \(M=_{\add}M'\) for any \(\Lambda\)-modules \(M\) and \(M'\).
We give a proof of Theorem \ref{main-theorem-2}.
\begin{proof}[Proof of Theorem \ref{main-theorem-2}]
  Let \(M\) be a \(B\)-module satisfying that \(\tilde{B}\Ind_G^{\tilde{G}}M\) is a support \(\tau\)-tilting \(\tilde{B}\)-module and \(\beta\) the block of \(kI_{\tilde{G}}(B)\) satisfying that
  \[
    1_{\tilde{B}}=\sum_{x\in[\tilde{G}/I_{\tilde{G}}(B)]}x1_{\beta}x^{-1},
  \]
  where \(1_{\tilde{B}}\) and \(1_{\beta}\) mean the respective identity elements of \(\tilde
  {B}\) and \(\beta\).
  By \cite[Proposition 4.7.]{https://doi.org/10.48550/arxiv.2301.04963}, we have that \(\tilde{B}\Ind_{G}^{\tilde{G}}M\cong \Ind^{\tilde{G}}_{I_{\tilde{G}}(B)}\beta \Ind_{G}^{I_{\tilde{G}}(B)}M\).
  Since the functor
  \[
    \begin{tikzcd}[row sep=0pt]
      \beta\lmod\ar[r] &\tilde{B}\lmod\\
      X\ar[r,mapsto]&\Ind_{I_{\tilde{G}}(B)}^{\tilde{G}} X
    \end{tikzcd}
  \]
  gives a Morita equivalence by \cite[Theorem 5.12.]{MR998775}, the \(\beta\)-module \(\beta \Ind_{G}^{I_{\tilde{G}}(B)}M\) is a support \(\tau\)-tilting \(\beta\)-module.
  Furthermore, by \cite[Proposition 4.3.(1)]{https://doi.org/10.48550/arxiv.2301.04963}, the restricted module \(\Res_{G}^{I_{\tilde{G}}(B)}\beta \Ind_{G}^{I_{\tilde{G}}(B)}M\) has a direct summand isomorphic to \(M\).
  Since the restricted module \(\Res_{G}^{I_{\tilde{G}}(B)}\beta \Ind_{G}^{I_{\tilde{G}}(B)}M\) is \(I_{\tilde{G}}(B)\)-invariant by \cite[Lemma 2.5.]{https://doi.org/10.48550/arxiv.2301.04963}, we get that
  \begin{equation*}
    \bigoplus_{\tilde{g}\in[I_{\tilde{G}}(B)/G]}\tilde{g}M\leq_{\add}\Res_{G}^{I_{\tilde{G}}(B)}\beta \Ind_{G}^{I_{\tilde{G}}(B)}M \leq_{\add} \Res_{G}^{I_{\tilde{G}}(B)}\Ind_{G}^{I_{\tilde{G}}(B)}M \cong \bigoplus_{\tilde{g}\in[I_{\tilde{G}}(B)/G]}\tilde{g}M.
  \end{equation*}
  Hence, we conclude that
  \begin{equation*}
    \Res_{G}^{I_{\tilde{G}}(B)}\beta \Ind_{G}^{I_{\tilde{G}}(B)}M=_{\add}\bigoplus_{\tilde{g}\in[I_{\tilde{G}}(B)/G]}\tilde{g}M.
  \end{equation*}
  Moreover, it holds that
  \begin{equation*}
    \beta\Ind_{G}^{I_{\tilde{G}}(B)}\Res_{G}^{I_{\tilde{G}}(B)}\beta \Ind_{G}^{I_{\tilde{G}}(B)}M
    =_{\add}\beta\Ind_{G}^{I_{\tilde{G}}(B)}\bigoplus_{\tilde{g}\in[I_{\tilde{G}}(B)/G]}\tilde{g}M
    =_{\add}\beta\Ind_{G}^{I_{\tilde{G}}(B)}M.
  \end{equation*}
  Hence, the restricted modules \(\Res_{G}^{I_{\tilde{G}}(B)}\beta \Ind_{G}^{I_{\tilde{G}}(B)}M\), which is additive equivalent to \(\bigoplus_{\tilde{g}\in[I_{\tilde{G}}(B)/G]}\tilde{g}M\) is support \(\tau\)-tilting \(B\)-module by \cite[Theorem 5.1.]{https://doi.org/10.48550/arxiv.2301.04963}.
  Furthermore, we get that the \(B\)-module \(M\) is \(\tau\)-rigid because it is a direct summand of the support \(\tau\)-tilting \(B\)-module \(\bigoplus_{\tilde{g}\in[I_{\tilde{G}}(B)/G]}\tilde{g}M\).

  On the other hand, let \(M\) be a \(\tau\)-rigid \(B\)-module such that \(\bigoplus_{\tilde{g}\in[I_{\tilde{G}}(B)/G]}\tilde{g}M\) is a support \(\tau\)-tilting \(kG\)-module.
  Since the \(B\)-module \(\bigoplus_{\tilde{g}\in[\tilde{G}/G]}\tilde{g}M\) is \(I_{\tilde{G}}(B)\)-invariant, we get that \(\tilde{B}\Ind_G^{\tilde{G}}M\), which is additive equivalent to \(\tilde{B}(\Ind_G^{\tilde{G}}(\bigoplus_{\tilde{g}\in[I_{\tilde{G}}(B)/G]}\tilde{g} M))\), is a support \(\tau\)-tilting \(\tilde{B}\)-module by \cite[Main Theorem 1.1.]{10.55937/sut/1670501315}.
\end{proof}

From now on, we use the following notation:
Let \(\Lambda\) be a finite-dimensional algebra and \(S, T\) simple \(\Lambda\)-modules.
We denote an indecomposable \(\Lambda\)-module \(M\) such that the composition length of \(M\) is equal to \(2\), \(\rad M \cong T\) and that \(M/\rad M \cong S\) by \(\begin{bmatrix}S\\T \end{bmatrix}\).
We remark that such \(M\) is unique up to isomorphism if \(\dim \Ext_{\Lambda}^1(S, T)=1\).

\begin{example}\label{example-1}
  Let \(k\) be an algebraically closed field of characteristic \(p=2\),
  \(G\) the alternating group \(A_4\) of degree \(4\) and \(\tilde{G}\) the symmetric group \(S_4\) of degree \(4\).
  Moreover, let \(S\) and \(T\) be non-isomorphic simple \(kG\)-modules to each other which are not isomorphic to the trivial \(kG\)-module,
  and \(S_2\) be a \(2\)-dimensional simple \(k\tilde{G}\)-module.
  As is well known, \(\sigma S\) is isomorphic to \(T\)
  and \(\sigma T\) is isomorphic to \(S\) for any \(\sigma \in \tilde{G}\setminus G\).
  In this setting,
  \(M:=k_G\oplus \begin{bmatrix} k_{G} \\S \end{bmatrix}\oplus \begin{bmatrix} k_{G} \\T \end{bmatrix}\)
  is a \(\tilde{G}\)-invariant support \(\tau\)-tilting \(kG\)-module.
  Hence, \(\Ind_G^{\tilde{G}}M\) is a support \(\tau\)-tilting \(k\tilde{G}\)-module by \cite[Main Theorem 1]{10.55937/sut/1670501315}.
  Also, let \(N_1:=k_G\oplus \begin{bmatrix} k_{G} \\S \end{bmatrix}\) and \(N_2:=k_G\oplus \begin{bmatrix} k_{G} \\T \end{bmatrix}\).
  Then \(N_1\) and \(N_2\) are support \(\tau\)-tilting \(kG\)-modules but are not \(\tilde{G}\)-invariant (hence we can not apply \cite[Main Theorem 1]{10.55937/sut/1670501315}).
  However, it holds that \(\bigoplus_{\sigma\in[\tilde{G}/G]}\sigma N_i=_{\add}M\)
  and in fact \(\Ind_G^{\tilde{G}}N_i\) is a support \(\tau\)-tilting \(k\tilde{G}\)-module
  for \(i=1, 2\).
\end{example}

\begin{remark}
  By Theorem \ref{main-theorem-1}, the set \((\Ind_G^{\tilde{G}})^{-1}(\sttilt k\tilde{G})\) is equal to the following set:
  \begin{equation}\label{set rig group}
    \left\{ M\in \taurigid kG \setmid \bigoplus_{\tilde{g}\in [\tilde{G}/G]}\tilde{g}M \in \sttilt kG \right\} .
  \end{equation}
  As a natural question, we wonder if the set \eqref{set rig group} coincides with the following set or not:
  \begin{equation}\label{set sta group}
    \left\{ M\in \stautilt kG \setmid \bigoplus_{\tilde{g}\in [\tilde{G}/G]}\tilde{g}M \in \sttilt kG \right\}.
  \end{equation}
  The set \eqref{set sta group} is contained in the set \eqref{set rig group} clearly, but the two sets do not coincide in general.
  In fact, in the same setting as Example \ref{example-1},
  let \(M:=\begin{bmatrix}S\\T \end{bmatrix}\).
  Then we can see that \(M\) is a \(\tau\)-rigid \(kG\)-module and that \(\bigoplus_{\tilde{g}\in[\tilde{G}/G]}\tilde{g}M\) is a support \(\tau\)-tilting \(kG\)-module by the calculations, but \(M\) is not a support \(\tau\)-tilting \(kG\)-module (see \cite[Proposition 1.8]{MR3461065}).
  On the other hand, the induced module \(\Ind_G^{\tilde{G}}M\) of \(M\) is a support \(\tau\)-tilting \(k\tilde{G}\)-module.
  Therefore, \(M\) is an element of the set \eqref{set rig group} but not the set \eqref{set sta group}.
  Also, since the principal blocks of \(kG\) and \(k\tilde{G}\) are themselves, respectively, the following two sets do not coincide in general:
  \begin{align}
     & \left\{ M\in \taurigid B \setmid \bigoplus_{\tilde{g}\in [I_{\tilde{G}}(B)/G]}\tilde{g}M \in \sttilt B \right\},\label{set rig block} \\
     & \left\{ M\in \stautilt B \setmid \bigoplus_{\tilde{g}\in [I_{\tilde{G}}(B)/G]}\tilde{g}M \in \sttilt B \right\}.\label{set sta block}
  \end{align}
\end{remark}



\Addresses

\end{document}